\documentclass[11pt]{article}
\usepackage[top=2.54cm, bottom=2.54cm, left=3cm, right=3cm]{geometry}
\usepackage[tbtags]{amsmath}
\usepackage{amssymb}
\usepackage{amsthm}
\usepackage{appendix}
\usepackage{latexsym}
\usepackage{mathrsfs}
\usepackage{xcolor}
\usepackage{wasysym}
\usepackage{fancyhdr}
\usepackage[colorlinks=true, linkcolor=red, citecolor=blue, urlcolor=magenta]{hyperref}
\usepackage{float}
\usepackage{graphicx}
\usepackage[numbers,sort&compress]{natbib}

\usepackage{enumerate}
\usepackage{verbatim}

\allowdisplaybreaks[4]

\newcommand{\eps}{\varepsilon}

\newtheorem{theorem}{Theorem}[section]

\newtheorem{proposition}[theorem]{Proposition}
\newtheorem{corollary}[theorem]{Corollary}
\theoremstyle{definition}

\newtheorem{conjecture}[theorem]{Conjecture}

\title{The logarithmic exponent of the Ramsey number $r(4,t)$\footnote{Center for Discrete Mathematics, Fuzhou University,
Fuzhou 350108, P.~R.~China. Email: {\tt linqizhong@fzu.edu.cn, 3079280593@qq.com}. Supported in part by the National Key R\&D Program of China (Grant No. 2023YFA1010202) and the NSFC (No.\ 12571361).}}

\author{Qizhong Lin \;\; and \;\; Qiuyang Zhang}

\date{}

\begin{document}

\maketitle

\begin{abstract}
We prove that 
\[
  r(4,t)\ge\frac{t^3}{(\log t)^{2+o(1)}}.
\]
This improves the lower bound of Mattheus and Verstra\"ete
(\emph{Ann. of Math.}, 2024) by a factor of \((\log t)^{2-o(1)}\).

\end{abstract}

\section{Introduction}
\label{sec:intro}

For integers $s,t\ge2$, the \emph{Ramsey number} $r(s,t)$ is
the least integer $n$ such that every $n$-vertex graph contains
either a complete graph $K_s$ or an independent set of size $t$.
Throughout this note, the clique parameter is fixed and $t\to\infty$
through the positive integers. All logarithms are natural.

For each fixed $s\ge3$, the classical upper bound of Ajtai, Koml\'os
and Szemer\'edi \cite{AjtaiKomlosSzemeredi1980}, in the refined form of
Li, Rousseau and Zang \cite{LiRousseauZang2001}, implies
\begin{equation}\label{eq:upper}
 r(s,t)\le C_s\frac{t^{s-1}}{(\log t)^{s-2}}
\end{equation}
with a constant $C_s>0$. At $s=3$, Kim \cite{Kim1995} proved the matching
lower bound $r(3,t)\ge ct^2/\log t$, thereby establishing
$r(3,t)=\Theta(t^2/\log t)$. At $s=4$, Mattheus and Verstra\"ete
\cite{MV2024} established the lower bound
$\Omega(t^3/(\log t)^4)$. Thus identifying the correct logarithmic
power in the lower bound is a separate issue from determining the
polynomial exponent.

Recently, Brada\v{c} \cite{Bradac2026} constructed a new family of
$K_s$-free graphs and proved that for every fixed $s\ge3$,
\begin{equation}\label{eq:bradac}
 r(s,t)\ge\Omega\!\left(\frac{t^{s-1}}{(\log t)^{2s-4}}\right)
 \qquad(t\to\infty).
\end{equation}
This matches the Erd\H{o}s--Szekeres upper bound up to polylogarithmic
factors. For $s=4$, \eqref{eq:bradac} recovers the polynomial exponent
$3$; for $s=5$, it gives $r(5,t)\ge ct^4/(\log t)^6$. The logarithmic
exponent $6$ is not sharp: the following input of OpenAI improves it
to $3+\delta$ for every fixed $\delta>0$.

We take as an external input the following statement from
\cite[Theorem~1.1]{OpenAI2026R5}: for every fixed $\delta>0$,
there exists $T_\delta$ such that
\begin{equation}\label{eq:H5}
 r(5,t)\ge\frac{t^4}{(\log t)^{3+\delta}}\qquad (t\ge T_\delta).
\end{equation}

Our result is as follows.

\begin{theorem}\label{thm:main}
There exists a constant $C>0$ such
that, for every $\eps>0$, there is $t_0(\eps)$ for which
\begin{equation}\label{eq:main-bounds}
 \frac{t^3}{(\log t)^{2+\eps}}
 \le r(4,t)\le C\frac{t^3}{(\log t)^2}
 \qquad(t\ge t_0(\eps)).
\end{equation}
Consequently,
\begin{equation}\label{eq:main}
 r(4,t)=\frac{t^3}{(\log t)^{2+o(1)}}
 \qquad(t\to\infty).
\end{equation}
\end{theorem}

The key observation is that a neighborhood in a $K_{k+1}$-free graph
is $K_k$-free. A suitable independence estimate forces a large
neighborhood whenever the graph itself is large and has small
independence number. Section~\ref{sec:bounds} states two estimates:
Shearer's 1995 theorem \cite{Shearer1995}, and the stronger local
independence bound of Li, Rousseau and Zang \cite{LiRousseauZang2001}.
They lead to two quantitative versions of the same descent inequality.
The first is sufficient for Theorem~\ref{thm:main}; the second yields
a sharper descent inequality, removing the $\log\log t$ factor at a fixed input exponent.

\section{Two independence estimates}
\label{sec:bounds}

For a graph $G$ let $\alpha(G)$ be its independence number,
$\Delta(G)$ its maximum degree, and $N_G(v)$ the neighborhood of
$v\in V(G)$. For a set \(S\subseteq V(G)\), let \(G[S]\) denote the subgraph of \(G\) 
induced by \(S\). As usual, $K_q$-free means that $G$ contains no copy
of $K_q$.

We retain the classical result in the maximum-degree form used in the
original descent argument.

\begin{theorem}[Shearer, 1995]
\label{thm:shearer}
For every fixed integer $q\ge4$, there exist constants $a_q>0$
and $D_q\ge3$ such that, for every integer $D\ge D_q$, every
$K_q$-free graph $G$ with $\Delta(G)\le D$ satisfies
\[
 \alpha(G)\ge a_q\,|V(G)|
      \frac{\log D}{D\log\log D}.
\]
\end{theorem}

A classical local independence estimate of Li, Rousseau and Zang
\cite{LiRousseauZang2001} gives the following stronger bound.
We quote its form in terms of the average degree of neighborhoods,
because that is what neighborhood descent uses directly.

\begin{theorem}[Li--Rousseau--Zang, 2001]
\label{thm:local-ind}
Let $G$ be a graph with $N$ vertices and average degree $d$.
Let $a\ge0$ be an integer. If the average degree of every subgraph
induced by a neighborhood is at most $a$, then
\[
 \alpha(G)\ge N f_{a+1}(d),
\]
where
\[
 f_m(x)=\int_0^1\frac{(1-u)^{1/m}}{m+(x-m)u}\,du.
\]
In particular, for $x>m\ge1$,
\[
 f_m(x)\ge \frac{\log(x/m)-1}{x}.
\]
\end{theorem}

Both statements keep the forbidden clique size fixed and allow $D$ to
be a degree cap rather than the actual maximum degree. The difference
is precisely the additional denominator $\log\log D$ in
Theorem \ref{thm:shearer}. Theorem~\ref{thm:local-ind} is used below only
as a classical published input.

\section{Comparison of adjacent clique orders}
\label{sec:comparison}

Write $r_k=r(k,t)$.
The following proposition separates the classical comparison from the
sharper one supplied by Theorem~\ref{thm:local-ind}.

\begin{proposition}\label{prop:descent}
Let $k\ge3$ be fixed. There are constants $c_k,\widetilde c_k>0$
and $t_k$ such that for every integer $t\ge t_k$ the following
statements hold.

\smallskip
\textbf{(i)} Using Shearer's theorem alone,
\begin{equation}\label{eq:classic-descent}
 r_k\ge c_k\frac{\log t}{t\log\log t}\,r_{k+1}.
\end{equation}

\textbf{(ii)} Using Theorem~\ref{thm:local-ind} together with the asymptotic
\(r(k,t)=t^{k-1+o(1)}\), we have
\begin{equation}\label{eq:strong-descent}
 r_k\ge\widetilde c_k\frac{\log t}{t}\,r_{k+1}.
\end{equation}
The constants depend only on $k$.
\end{proposition}

\begin{proof}
Set $n=r_{k+1}-1$ and $D=r_k-1$.
By the definition of \(r_{k+1}\), there exists a \(K_{k+1}\)-free
graph \(G\) on \(n\) vertices with \(\alpha(G)\le t-1\).

For every \(v\in V(G)\), the subgraph \(G[N_G(v)]\) is \(K_k\)-free.
If \(|N_G(v)|\ge r_k$, then, by the definition of $r_k\), the
neighborhood \(N_G(v)\) contains an independent set of size \(t\).
This set is also independent in \(G\), contradicting
\(\alpha(G)\le t-1\). Hence \(|N_G(v)|\le r_k-1=D\), and therefore
\begin{equation}\label{eq:degree-cap}
 \Delta(G)\le D.
\end{equation}
Moreover, \(\alpha(G[N_G(v)])\le \alpha(G)\le t-1\).

The elementary Ramsey bounds give
\begin{equation}\label{eq:D-range}
 t-1\le D\le\binom{t+k-2}{k-1}-1=O_k(t^{k-1}).
\end{equation}
In particular, \(D\to\infty\) as \(t\to\infty\), and
\begin{equation}\label{eq:log-ratio}
 \log D\ge\log t-O(1),\qquad
 \log\log D\le\log\log t+O_k(1).
\end{equation}

\medskip
\textbf{(i)} Applying Theorem~\ref{thm:shearer} with forbidden
clique size \(q=k+1\) to \(G\) yields
\[
\alpha(G)\ge a_{k+1} n\frac{\log D}{D\log\log D}.
\]
Since \(\alpha(G)\le t-1\), we obtain
$t-1\ge a_{k+1} n\frac{\log D}{D\log\log D}$,
or equivalently
\begin{equation}\label{eq:exact-classical}
 n\le\frac{t-1}{a_{k+1}}\,
      D\frac{\log\log D}{\log D}.
\end{equation}
By \eqref{eq:log-ratio}, for sufficiently large \(t\) the last
fraction is \(O_k(\log\log t/\log t)\).
Substituting \(n=r_{k+1}-1\) and \(D=r_k-1\), and absorbing the
additive \(1\), we obtain that there is a constant $c_k>0$ such that
\[
r_k\ge c_k\frac{\log t}{t\log\log t}\,r_{k+1}.
\]
This proves \eqref{eq:classic-descent}.

\medskip
\textbf{(ii)} Using Theorem~\ref{thm:local-ind} together with the asymptotic \(r(k,t)=t^{k-1+o(1)}\), 
we obtain the stronger descent inequality.

Note that for every \(v\in V(G)\), the neighborhood \(N_G(v)\) induces a \(K_k\)-free subgraph 
with independence number at most \(t-1\). We claim that 
\(\Delta(G[N_G(v)])\le a:=r_{k-1}-1\). Indeed, if some 
\(u\in N_G(v)\) had degree at least \(r_{k-1}\) in \(G[N_G(v)]\), 
then \(N_G(u)\cap N_G(v)\) would be \(K_{k-1}\)-free and have 
at least \(r_{k-1}\) vertices, hence would contain an independent 
set of size \(t\), contradicting \(\alpha(G)\le t-1\). Thus the 
average degree of \(G[N_G(v)]\) is at most \(a\) for every \(v\), 
so the hypothesis of Theorem~\ref{thm:local-ind} is satisfied. 

Since \(f_{a+1}\) is decreasing and \(d\le D\), we have 
\(f_{a+1}(d)\ge f_{a+1}(D)\). As \(D>a+1\) 
for all sufficiently large \(t\), by Theorem~\ref{thm:local-ind}, we obtain
\[
t-1\ge \alpha(G)\ge n f_{a+1}(D)
      \ge n\frac{\log (D/(a+1))-1}{D}.
\]
Dividing by \(t-1\) and using \(t-1<t\), it follows that
\[
D\ge \frac{n}{t}\bigl(\log(D/(a+1))-1\bigr).
\]
By the recent result of Brada\v{c} \cite{Bradac2026},
\(r(k,t)=t^{k-1+o(1)}\) for every fixed \(k\). Hence
$\frac{r_k}{r_{k-1}}=t^{1+o(1)}$.
In particular, for all sufficiently large \(t\),
$r_k\ge 2t^{1/2}r_{k-1}$.
Therefore
$\frac{D}{a+1}=\frac{r_k-1}{r_{k-1}}\ge  t^{1/2}$,
and so
$\log\frac{D}{a+1}-1\ge \frac{1}{4}\log t$
for all sufficiently large \(t\). Substituting \(n=r_{k+1}-1\) and
\(D=r_{k}-1\), and absorbing the additive \(1\), we obtain $ r_k\ge\widetilde c_k\frac{\log t}{t}\,r_{k+1}$ for some constant $\widetilde c_k>0$, which proves \eqref{eq:strong-descent}.
\end{proof}

Two features of the comparison are worth noting. It preserves the
independence-set parameter $t$; only the forbidden clique size falls
by one. Also, the stronger comparison does not require any new
geometric construction: its improvement comes solely from the
stronger input from Theorem~\ref{thm:local-ind}.

\section{Consequences for \texorpdfstring{$r(4,t)$}{r(4,t)} and fixed orders}
\label{sec:consequences}

\begin{proof}[Proof of Theorem~\ref{thm:main}]
By \eqref{eq:classic-descent} with \(k=4\), there is a constant \(c>0\)
such that, for all sufficiently large \(t\),
\begin{equation}\label{eq:r4-classic}
 r(4,t)\ge c\,\frac{\log t}{t\log\log t}r(5,t).
\end{equation}
Fix \(\eps>0\) and choose \(\delta=\eps/2\) in \eqref{eq:H5}.
For all sufficiently large \(t\), this gives
\begin{equation}\label{eq:r4-with-loss}
 r(4,t)\ge c\frac{t^3}{(\log t)^{2+\eps/2}\log\log t}.
\end{equation}
Since \(\log\log t=o((\log t)^{\eps/2})\), the right-hand side is
at least \(t^3/(\log t)^{2+\eps}\) for all sufficiently large \(t\).
This proves the lower bound in \eqref{eq:main-bounds}.
The classical upper bound \eqref{eq:upper} with \(s=4\) gives the
other bound, and hence \eqref{eq:main}.
\end{proof}

The two independence estimates yield different quantitative lower bounds at the same fixed \(\delta\).
\begin{equation}\label{eq:comparison}
 \begin{aligned}
 \text{Shearer 1995:}\quad
 &r(4,t)\ge c\,\frac{t^3}{(\log t)^{2+\delta}\log\log t},\\
 \text{Li--Rousseau--Zang 2001:}\quad
 &r(4,t)\ge\widetilde c\,\frac{t^3}{(\log t)^{2+\delta}}.
 \end{aligned}
\end{equation}
These are estimates at the \emph{same fixed} \(\delta\) coming from the
same \(r(5,t)\) input. Both lead to the logarithmic exponent \(2+o(1)\);
the Li--Rousseau--Zang transfer is more efficient, but it is not a
constant-factor asymptotic for \(r(4,t)\).

\begin{corollary}
\label{cor:sharp}
Fix integers \(3\le m<s\). If
\[
 r(s,t)=\frac{t^{s-1}}{(\log t)^{s-2+o(1)}},
\]
then
\[
 r(m,t)=\frac{t^{m-1}}{(\log t)^{m-2+o(1)}}.
\]

\end{corollary}

\begin{proof}
Fix \(\eta>0\). The hypothesis gives
\(r(s,t)\ge t^{s-1}/(\log t)^{s-2+\eta}\) for all
sufficiently large \(t\).
Applying Proposition~\ref{prop:descent}(i) successively for
\(k=m,m+1,\ldots,s-1\) and multiplying the resulting inequalities,
we obtain
$r(m,t)\ge c_{m,s}\,r(s,t)
 \left(\frac{\log t}{t\log\log t}\right)^{s-m}$
for some constant \(c_{m,s}>0\). Substituting the hypothesis gives
$r(m,t)\ge c_{m,s}
 \frac{t^{m-1}}{(\log t)^{m-2+\eta}(\log\log t)^{s-m}}$.
Since \(m\) and \(s\) are fixed, \((\log\log t)^{s-m}\) is smaller than
any fixed positive power of \(\log t\). It follows that for every
\(\eps>0\), \(r(m,t)\ge t^{m-1}/(\log t)^{m-2+\eps}\) for all
sufficiently large \(t\), after taking \(\eta<\eps\).
The matching upper bound at the logarithmic-exponent level follows
from \eqref{eq:upper}.
\end{proof}

\section{Discussion and limitations}
\label{sec:discussion}

Theorems~\ref{thm:shearer} and \ref{thm:local-ind} give comparable
independence estimates. In Proposition~\ref{prop:descent},
Theorem~\ref{thm:shearer} yields the classical descent inequality
\eqref{eq:classic-descent} in part~(i), while
Theorem~\ref{thm:local-ind} yields the stronger descent inequality
\eqref{eq:strong-descent} in part~(ii). The classical descent already
suffices, together with \eqref{eq:H5}, for Theorem~\ref{thm:main}; the
stronger descent removes the extra \(\log\log t\) factor and is
relevant to quantitative comparisons.

Under the assumed lower bound \eqref{eq:H5} for \(r(5,t)\), the
argument above determines the logarithmic exponent of \(r(4,t)\)
almost exactly: $r(4,t)=\frac{t^3}{(\log t)^{2+o(1)}}$.
The upper bound has logarithmic exponent \(2\), while the lower bound
has exponent \(2+\eps\) for every fixed \(\eps>0\); the two bounds
differ only by a possible factor \((\log t)^{o(1)}\). It is widely believed that the
upper bound is sharp up to a constant factor.  We record this
expectation as a precise conjecture.
\begin{conjecture}
$r(4,t)=\Theta\!\left(\frac{t^3}{(\log t)^2}\right)$.
\end{conjecture}

\bibliographystyle{plain}

\end{document}